\documentclass[a4paper,12pt]{article}
\usepackage{amsmath, amsthm, amstext}
\usepackage{amssymb, array, amsfonts}
\usepackage[english]{babel}
\usepackage{bbding}

\numberwithin{equation}{section}

\def \er{\varepsilon}

\def \ka{\varkappa}

\renewcommand{\l}{\left}
\renewcommand{\r}{\right}

\def \C{\mathbb{C}}
\def \N{\mathbb{N}}
\def \R{\mathbb{R}}

\def \A{\mathcal{A}}

\def \zbar {\bar{z}}

\newcommand{\beq}{\begin{equation}}
\newcommand{\eeq}{\end{equation}}
\newcommand{\one}{\mathbf{1}}

\def\im{\operatorname{Im}}

\newcommand{\eqdef}{\stackrel{\rm def}{=\kern-3.6pt=}}

\theoremstyle{plain}
\newtheorem{theorem}{\bf Theorem}[section]
\newtheorem{lemma}[theorem]{\bf Lemma}

\newtheorem{prop}[theorem]{\bf Proposition}

\theoremstyle{definition}
\newtheorem{defi}[theorem]{\bf Definition}

\theoremstyle{remark}
\newtheorem{remark}[theorem]{\bf Remark}

\renewcommand{\le}{\leqslant}
\renewcommand{\ge}{\geqslant}
\newcommand{\dist}{\mathop{\mathrm{dist}}\nolimits}

\renewcommand{\qed}{\vrule height7pt width5pt depth0pt}

\title{Polynomials of almost normal arguments in
 $C^*$-algebras}

\author{N. Filonov, I. Kachkovskiy\thanks{Steklov Institute, 
St. Petersburg, and King's College London. 
The first author was supported by RFBR Grant 11-01-00324-a. 
The second author was supported by King's Annual Fund Studentship 
and King's Overseas Research Studentship.}}
\date{}
\begin{document}
\maketitle
\begin{abstract}
The functional calculus for normal elements in $C^*$-algebras is 
an important tool of analysis. 
We consider polynomials $p(a,a^*)$ for elements $a$ with 
small self-commutator norm $\|[a,a^*]\| \le \delta$ and 
show that many properties of the functional calculus  
are retained modulo an error of order $\delta$.

{\it 2010 Mathematics Subject Classification:} 47A60, 46L05, 11E25.

{\it Keywords:} $C^*$-algebras, functional calculus, self-commutator,
polynomials, Positivstellensatz, pseudospectrum.
\end{abstract}

\section{Introduction}
Let $a$ be a {\it normal} element of a unital $C^*$-algebra $\A$.
It is well known that there exists a unique $C^*$-algebra homomorphism
$$
C(\sigma (a)) \to \A, \quad f\mapsto f(a)
$$
from the algebra of continuous functions on the spectrum $\sigma(a)$ into $\A$
such that $f(z) = z$ is mapped into $a$,
$\sigma(f(a)) = f(\sigma(a))$, and
\beq
\label{1}
\|f(a)\| = \max_{z\in\sigma (a)} |f(z)| 
\eeq
(see, for example, \cite{Dix}).
It is called the functional calculus for normal elements and is widely used in analysis.

The aim of the present paper is to introduce an analogue of  functional calculus for 
``almost normal'' elements. More precisely, we shall always be assuming that
\beq
\label{2}
\|a\| \le 1, \quad \|[a, a^*]\| \le \delta 
\eeq
with a small $\delta$.
We restrict the considered class of functions to polynomials 
in $z$ and $\zbar$ and show that some important properties of the functional calculus hold up to an error of order $\delta$.

If $aa^* \neq a^*a$ then the polynomials of $a$ and $a^*$ are, 
in general, not uniquely defined. 
We fix the following definition. For a polynomial 
\beq
\label{p_def}
p(z,\zbar)=\sum\limits_{k,l} p_{kl} z^k \zbar^l
\eeq
let
\beq
\label{pa_def}
p(a,a^*)=\sum\limits_{k,l} p_{kl} a^k (a^*)^l.
\eeq
It is clear that the map $p\mapsto p(a,a^*)$
is linear and involutive, that is
$\overline{p}(a,a^*)=p(a,a^*)^*$ where $\bar{p}(z,\zbar)=\sum\bar{p}_{lk}z^k \zbar^l$.
Using the inequality 
$\|[a, b^m]\| \le m \|b\|^{m-1} \|[a,b]\|$
and \eqref{2}, one can easily show that the map $p\mapsto p(a,a^*)$
is ``almost multiplicative'',
\beq 
\label{appr_mult}
\|p(a,a^*)q(a,a^*)-(pq)(a,a^*)\|\leqslant C(p,q)\,\delta
\eeq 
where
$$
C(p,q)=\sum\limits_{k,l,s,t} l s \l|p_{kl}\r|\l|q_{st}\r|.
$$
It takes much more effort to obtain an estimate of the norm $\|p(a,a^*)\|$.
In the case of an analytic polynomial $p(z) = \sum_k p_k z^k$, according to 
the von Neumann inequality, 
$$
\|p(a)\|\leqslant \max_{|z|\le 1}|p(z)| =: p_{\max}
$$
where it is only assumed that $\|a\| \le 1$ (see, for example, \cite[I.9]{SF}).

Our main results are as follows.
\begin{theorem}
\label{norm_th}
Let $p$ be a polynomial \eqref{p_def}. There exists a constant $C(p)$ 
such that the estimate 
\beq
\label{neum_gen}
\|p(a,a^*)\|\leqslant p_{\max}+C(p)\delta
\eeq
holds for all $a$ satisfying \eqref{2}.
Here $p(a,a^*)$ is defined by \eqref{pa_def}, and 
$p_{\max}=\max\limits_{|z|\leqslant 1}|p(z,\zbar)|$.
\end{theorem}

If $a$ is normal and $f$ is a continuous function then 
the functional calculus gives the following more precise estimate,
\beq
\label{norm}
\|f(a)\|=\max\limits_{z\in \sigma(a)}|f(z)|.
\eeq
If $a\in\A$ and $\lambda_j\not\in\sigma(a)$, $j=1,\ldots,m-1$, 
then there exists $R_j>0$ such that 
\beq
\label{a}
\|(a-\lambda_j)^{-1}\| \leqslant R_j^{-1},\mbox{ }j=1,\ldots,m-1.
\eeq
The following theorem gives an analogue of \eqref{norm} for an almost normal $a$.

\begin{theorem}
\label{t34}
Let $a\in\A$ satisfy \eqref{2} and \eqref{a}, and let the set
\beq
\label{s}
S=\{z\in\C\colon |z|\leqslant 1,\mbox{ }
|z-\lambda_j|\geqslant R_j,\mbox{ }j=1,\ldots,m-1\}
\eeq
be nonempty. 
For each $\er>0$ and each polynomial $p$ defined by \eqref{p_def} there exists a constant $C(p,\er)$ independent of $a$ such that
$$
\|p(a,a^*)\| \leqslant \max_{z\in S} |p(z,\zbar)| + \er + C(p,\er) \delta.
$$

\end{theorem}

Note that, under the conditions of Theorem \ref{t34}, the set $S$  is a unit disk with $m-1$ ``holes'' which contains $\sigma(a)$. 

Finally, assume again that $a$ is normal and $\mu \notin f(\sigma(a))$.
Then the functional calculus implies that the 
element $(f(a)-\mu)$ is invertible and
\beq
\label{5}
\left\| (f(a)-\mu)^{-1} \right\| = 
\frac{1}{\dist \left(\mu, f(\sigma(a))\right)}.
\eeq
The equality \eqref{5} also admits the following approximate analogue 
with $\sigma(a)$ replaced by $S$ and $f(\sigma(a))$ by $p(S)$, 
where $p(S)$ is the image of $S$ 
under $p$ considered as a map from $\C$ to $\C$.
\begin{theorem}
\label{main}
Let $S$ be defined by \eqref{s}, and let $p$ be a polynomial \eqref{p_def}.
Then for each $\varepsilon>0$ and $\ka>0$ there exist constants
$C(p,\ka,\er)$, $\delta_0(p,\ka,\er)$ 
such that for all $\delta<\delta_0(p,\ka,\er)$
and for all $\mu\in\C$ satisfying $\dist(\mu,p(S))\ge \ka$
the estimate
$$
\|(p(a,a^*) - \mu\one)^{-1}\| \leqslant \ka^{-1} + \varepsilon
+ C(p,\ka,\er)\delta
$$
holds for all $a\in\A$ satisfying \eqref{2} and \eqref{a}.
\end{theorem}

The authors' interest to the subject was drawn by its relation with 
Huaxin Lin's theorem (see \cite{L,FR}). 
It says that if $a$ is an $n\times n$-matrix satisfying \eqref{2}, 
then the distance from $a$ to the set of normal matrices 
is estimated by a function $F(\delta)$ such that
$F(\delta)\to 0$ as $\delta \to 0$ uniformly in $n$. 
This result implies Theorems \ref{norm_th}--\ref{main} 
with $\delta$ replaced by $F(\delta)$ 
in the right hand side. 
By homogenuity reasons, $F(\delta)$ can not decay faster 
than $C\delta^{1/2}$ as $\delta\to0$. 
Therefore this approach gives weaker 
results in terms of power of $\delta$.
Also, our results hold in any unital 
$C^*$-algebra, while the infinite-dimensional versions 
of Lin's theorem require additional index type assumptions on $a$ 
(see, for example, \cite{FR}).

Our proofs are based on certain representation theorems for positive polynomials.
If a real polynomial of $x_1$, $x_2$ is non-negative 
on the unit disk $\{ x : x_1^2 + x_2^2 < 1\}$
then, by a result of \cite{Sch}, it admits a representation
\beq
\label{sos_ref}
\sum_j r_j(x)^2 + \left(1-x_1^2-x_2^2\right) \sum_j s_j(x)^2
\eeq
with real polynomials $r_j$ and $s_j$ (see Proposition \ref{s_c} below).
Representations similar to \eqref{sos_ref} are  usually referred to as {\it Positivstellensatz}. 
We also make use of Positivstellensatz for polynomials positive on the sets \eqref{s}.
The corresponding results for sets bounded by arbitrary algebraic curves
were obtained in \cite{Cas,Put,Sch2,Sch}. 

In order to prove Theorem \ref{main}, we need uniform with respect to $\mu$ estimates for polynomials appearing in Positivstellensatz-type representations. 
In order to obtain the estimates, we use the scheme introduced in \cite{S,NS}.

The authors thank Dr. A. Pushnitski and the referee for valuable comments.

\section{Proofs of the main results}
The proofs of all three theorems consist of two parts. This section is devoted 
to the ``operator-theoretic'' part, which is essentially based on Lemma 
\ref{review}. The ``algebraic'' part is the existence of representations \eqref{sos_cstar}
for the polynomials \eqref{q_def}, \eqref{qq2}, \eqref{qq} 
which is discussed in Section 3.
\label{spec_sect}
\subsection{Positive elements of $C^*$-algebras}
\label{pos}
Recall that a Hermitian element $b\in\A$ is called {\it positive} 
($b\ge 0$) if one of the following two equivalent conditions holds
(see, for example, \cite[\S 1.6]{Dix}):
\begin{enumerate}
\item
$\sigma(b)\subset[0,+\infty).$
\item
$b=h^*h$ for some $h\in\A$.
\end{enumerate}
The set of all positive elements in $\A$ is a cone: 
if $a,b\ge 0$, then $\alpha a+\beta b\ge 0$
for all real $\alpha,\beta\ge 0$. 
There exists a partial ordering on the set of Hermitian elements of $\A$: 
$a\le b$ iff $b-a\ge 0$. 
For a Hermitian $b$,
\beq
\label{positive}
-\|b\|\one\le b\le \|b\|\one
\eeq and, 
moreover, if $0\le b\le \beta\one$, $\beta\in\R$,
then $\|b\|\le\beta$.
The following fact is also well known.

\begin{prop}
\label{l31}
Let $h\in\A$, $\rho >0$.
Then $h^*h \ge \rho^2\one$ if and only if
the element $h$ is invertible and  $\|h^{-1}\| \le \rho^{-1}$.
\end{prop}

Our proofs use the following simple lemma.

\begin{lemma}
\label{review}
Let $a\in\A$ satisfy \eqref{2}, and let
\beq
\label{sos_cstar}
q=\sum\limits_{j=0}^N r_j^2+\sum_{i=0}^{m-1}\l(\sum\limits_{j=0}^N r_{ij}^2\r) g_i,
\eeq
where $r_j$, $r_{ij}$, $g_i$ are real-valued polynomials of the form \eqref{p_def}.
Assume that $g_i(a,a^*)\ge 0$, $i=0,\ldots,m-1$. Then
$$
q(a,a^*)\ge -C\delta\one
$$
with some non-negative constant $C$ depending on $r_j$, $r_{ij}$, $g_j$.
\end{lemma}
\begin{proof}
Note that $q$ is real-valued, so $q(a,a^*)$ is self-adjoint. Since $g_{i}(a,a^*)\ge 0$,
we have $g_i(a,a^*)=b_i^* b_i$ for some $b_i\in\mathcal A$. Then
$$
r_{ij}(a,a^*)g_{i}(a,a^*)r_{ij}(a,a^*)=(b_i r_{ij}(a,a^*))^*(b_i r_{ij}(a,a^*))\ge 0.
$$
We also have $r_j(a,a^*)^2\ge 0$. From \eqref{appr_mult}, we have
$$
\|q(a,a^*)-\sum\limits_j r_j(a,a^*)^2-\sum\limits_{i,j}r_{ij}(a,a^*)g_{i}(a,a^*)r_{ij}(a,a^*)\|\le C'\delta,
$$
and now the proof is completed by using \eqref{positive}.
\end{proof}
\subsection{Proofs of Theorems \ref{norm_th}--\ref{main}}
{\it Proof of Theorem }\ref{norm_th}.
Proposition \ref{s_c} below implies that the polynomial
\beq
\label{q_def}
q(z,\zbar)=p_{\max}^2-|p(z,\zbar)|^2
\eeq
admits a representation \eqref{sos_cstar} with $m=1$, $g_0(z,\zbar)=1-|z|^2$ because, by the definition of $p_{\max}$, the polynomial $q$ is non-negative on the unit disk.

Let us apply Lemma \ref{review} to $q$. 
By \eqref{2}, we have $g_0(a,a^*) = \one - aa^* \ge 0$. 
Therefore 
$$
q(a,a^*)\ge -C_1(p) \delta\one
$$
from which, using \eqref{q_def} and \eqref{appr_mult}, we get
$$
p_{\max}^2 \one -p(a,a^*)^* p(a,a^*)\ge -C_2(p)\delta\one,
$$
$$
p(a,a^*)^* p(a,a^*) \leqslant \l(p_{\max}^2 + C_2(p) \delta\r)\one
$$
and
$$
\|p(a,a^*)\| \leqslant p_{\max}+\frac{C_2(p)\delta}{2 p_{\max}}.\mbox{ }\qed
$$

\noindent{\it Proof of Theorem }\ref{2}. 
By Theorem \ref{Putinar}, the polynomial
\beq
\label{qq2}
q(z,\zbar) = p_{\max}^2 + \er p_{\max} - |p(z,\zbar)|^2
\eeq
admits a representation \eqref{sos_cstar} with
\beq
\label{ggg}
g_0(z,\zbar)=1-|z|^2, \quad g_i(z,\zbar)=|z-\lambda_i|^2-R_i^2,\quad i=1,\ldots,m-1,
\eeq
because it is strictly positive on the set $S$. Note that
\beq
\label{sss}
S=\{z\in\C\colon g_i(z,\zbar)\ge 0,\mbox{ }i=0,\ldots,m-1\}.
\eeq

Proposition \ref{l31} and \eqref{a}  
imply 
$$
g_i(a,a^*) = (a-\lambda_i\one)(a-\lambda_i\one)^* - R_i^2\one \ge 0 ,
$$ 
so we can again apply Lemma \ref{review}. 
Using \eqref{appr_mult}, we obtain
$$
q(a,a^*)\ge -C_1 \delta \one,\quad C_1>0,
$$
$$
p(a,a^*)p(a,a^*)^*\le \l(p_{\max}^2 + \er p_{\max} + C_2(p,\er)\delta\r)\one,
$$
and
$$
\|p(a,a^*)\| \le p_{\max}
\sqrt{1+\frac{\er}{p_{\max}}+\frac{C_2(p,\er)\delta}{p_{\max}^2}}
\le p_{\max} + \er + \frac{C_2(p,\er)\delta}{p_{\max}}.\,\,\qed
$$

\noindent{\it Proof of Theorem }\ref{main}. 
Fix $\gamma>0$. 
By Theorem \ref{Putinar}, the polynomial
\beq
\label{qq}
q(z,\zbar)=|p(z,\zbar)-\mu|^2 - \ka^2 + \gamma.
\eeq
also admits a representation \eqref{sos_cstar} with the same $g_i$ given by \eqref{ggg}. This is because, by the definitions of $\mu$ and $\ka$, 
we have $q(z,\zbar)>0$ for all $z\in S$. 
Since $g_i(a,a^*)\ge 0$, Lemma \ref{review} implies
$$
q(a,a^*)\ge -C\delta\one, \quad C>0.
$$
Using \eqref{qq} and \eqref{appr_mult}, we obtain
\beq
\label{p_est}
(p(a,a^*)-\mu\one)^* (p(a,a^*)-\mu\one) \geqslant 
\left(\ka^2 - \gamma-C'\delta\right) \one .
\eeq
Let us choose $\gamma$ and $\delta_0$ such that
$\gamma + C'\delta\le \ka^2/2$.
Now, \eqref{p_est} and Proposition \ref{l31} give
$$
\|(p(a,a^*)-\mu\one)^{-1}\| \leqslant 
\l(\ka^2-\gamma-C'\delta\r)^{-1/2}\leqslant \ka^{-1}+\frac{\gamma}{\ka^2}+\frac{C'\delta}{\ka^2}.
$$
Choosing $\gamma\leqslant \varepsilon \ka^2$, we obtain the required inequality with $\ka^{-2}C'$ instead of $C$.

The constant $C'$, in general, depends on $p, \ka, \gamma$, and $\mu$. 
Let us show that the theorem holds with $C$ independent of $\mu$. 
For $|\mu|\geqslant \|p(a,a^*)\|+\ka$ it is obvious 
as
$$
\l\|(p(a,a^*)-\mu\one)^{-1}\r\|\leqslant 
\frac{1}{|\mu|-\|p(a,a^*)\|}\leqslant \ka^{-1}.
$$
Thus we can restrict the consideration to the compact set
$$
M = \{\mu\in\C\colon |\mu|\leqslant \|p(a,a^*)\|+\ka,\,
\dist(\mu,p(S))\ge\ka\}.
$$
The estimate $q(z,\zbar)\geqslant \gamma$ holds for all $\mu\in M$. 
The number $N$ of the polynomials $r_j$ and $r_{ij}$ as well as their powers and 
coefficients are bounded uniformly on $M$ by Remark \ref{constr}.
Since $C'$ depends only on these parameters, $C$ may be chosen independent of $\mu$.
\qed

\subsection{Corollaries and remarks}
\begin{remark}
As mentioned in the beginning of the section, the proofs rely on the existence 
of representations of the form \eqref{sos_cstar} for certain polynomials. 
In addition, we need continuity of such a representation with respect 
to the parameter $\mu$ to establish Theorem \ref{main}. 
We are also interested in the possibility of explicitly 
computing the constants $C$ and $\delta_0$, 
which may be important in applications. 
It is clearly possible if we have explicit formulae for the polynomials 
in \eqref{sos_cstar}. 
We show below that this can be done in Theorems \ref{2} and 
\ref{main} (see Remark \ref{constr}).
\end{remark}

\begin{remark}
\label{r32}
In general, it is not possible to find a constant $C$ in Theorem \ref{norm_th} 
which would work for all polynomials $p$.
As an example, consider $\A=M_2(\C)$, 
$$
a=\l(\begin{matrix}0&\sqrt{\delta}\\ 0&0 \end{matrix}\r), \quad 0<\delta<1.
$$
It is clear that $a$ satisfies \eqref{2}. 
Let $\er<1$. 
There exists a continuous function $f$ such that 
$f(z)=-1/z$ whenever $|z|\ge \er$ 
and $|f(z)|\le 1/\er$ for $|z|\le 1$. 
There also exists a 
polynomial $q(z,\zbar)$ such that $|q(z,\zbar)-f(z)|\le \er$ for $|z|\le 1$. 
Now, let
$$
p(z,\zbar)=\frac{1}{\er}\l(z+z^2q(z,\zbar)\r).
$$
Then $p_{\max}\le 2 + \er^2$, but 
$p(a,a^*) = a/\er$ and $\|p(a,a^*)\|=\sqrt{\delta}/\er$. 
Taking $\er$ small, we see that \eqref{neum_gen} can not hold 
with a $C$ independent of $p$.
\end{remark}

\begin{prop}
Under the assumptions of Theorem {\rm\ref{t34}}, there exists a constant
$C(p,\er)$ such that
$$
\|\im p(a,a^*)\| \leqslant \max_{z\in S} |\im p(z,\zbar)| + \er + C(p,\er) \delta .
$$
\end{prop}

\begin{proof}
It suffices to apply Theorem {\rm\ref{t34}} to the polynomial
$q(z,\zbar) = \frac{p(z,\zbar) - \overline{p(z,\zbar)}}{2i}$.
\end{proof}

In other words, if the values of $p$ on $S$
are almost real, then the element $p(a,a^*)$ itself is almost self-adjoint.

\begin{prop}
Under the assumptions of Theorem {\rm \ref{t34}}, 
there exists a constant $C(p,\er)$ such that
\beq
\label{uni1}
\| p(a,a^*)p(a,a^*)^* - \one\| \leqslant 
\max_{z\in S} \l||p(z,\zbar)|^2-1\r| + \er + C(p,\er) \delta,
\eeq
\beq
\label{uni2}
\|p(a,a^*)^*p(a,a^*) - \one\| \leqslant 
\max_{z\in S} \l||p(z,\zbar)|^2-1\r| + \er + C(p,\er) \delta.
\eeq
\end{prop}

\begin{proof}
It is sufficient to apply Theorem \ref{t34} to the polynomial 
$q(z,\zbar) = \l|p(z,\zbar)\r|^2-1$ and use \eqref{appr_mult}.\,\qedhere
\end{proof}

\begin{remark}
Denote the right hand side of \eqref{uni1}, \eqref{uni2} by $\gamma$. 
If $\gamma<1$ then 
$$
(1-\gamma)\one \le p(a,a^*)^*p(a,a^*)\le (1+\gamma)\one
$$
and
$$
(1-\gamma)\one\le p(a,a^*)p(a,a^*)^*\le (1+\gamma)\one,
$$
which implies that $p(a,a^*)$ and $p(a,a^*)^*p(a,a^*)$ are invertible. 
The element 
$$
u=p(a,a^*)\l(p(a,a^*)^*p(a,a^*)\r)^{-1/2}
$$
is unitary (because it is invertible and $uu^*=1$) and close to $u$,
$$
\|p(a,a^*)-u\|\le \sqrt{1+\gamma}\l(\frac{1}{\sqrt{1-\gamma}}-1\r)\to 0\quad\text{as}\quad\gamma\to 0.
$$
Thus if the absolute values of $p$ on $S$ are close to 
$1$ then $p(a,a^*)$ is close to a unitary element.
\end{remark}

\begin{defi}
The set
$$
\sigma_{\er}(a)=\{\lambda\in\C\colon \|(a-\lambda\one)^{-1}\|>1/\er\}\cup \sigma(a)
$$
is called the $\er$-{\it{pseudospectrum}} of the element $a\in\A$.
\end{defi}
\noindent Its main properties are discussed, for example, in \cite[Ch.~9]{BD}. 
Note that, under the assumptions of Theorem \ref{main},
$\sigma_{\er}(a) \subset {\cal O}_{\er}(S)$ for all $\er>0$,
where ${\cal O}_{\er}(S)$ is the $\er$-neighbourhood of $S$.
If $a$ is normal then
$$
\sigma_{\ka}(p(a,a^*)) = {\cal O}_{\ka} \l(p(\sigma(a))\r),\quad \ka>0.
$$
The following statement is Theorem \ref{main} reformulated in these terms.
\begin{prop}
\label{pseudo}
Under the assumptions of Theorem {\rm \ref{main}}, for all $\er>0$ and $\ka>0$ 
there exist $C(p,\ka,\er)$ and $\delta_0(p,\ka,\er)$ such that
$$
\sigma_{\ka'}(p(a,a^*))\subset {\cal O}_{\ka}(p(S)),\quad
\forall \delta<\delta_0(p,\ka,\er),
$$
where $(\ka')^{-1} = \ka^{-1}+\varepsilon+C(p,\ka,\er)\delta$.
\end{prop}

\begin{proof}
Assume that $\dist(\mu,p(S)) \geqslant \ka$. 
By Theorem \ref{main},
$\|(p(a,a^*)-\mu\one)^{-1}\| \leqslant (\ka')^{-1}$
and, consequently, $\mu\notin \sigma_{\ka'}\l(p(a,a^*)\r)$.
\end{proof}

\section{Representations of non-negative polynomials}
This section is devoted to a special case of the following theorem, which is 
often called {\it Putinar's Positivestellensatz}. As usual, we denote the ring of 
real polynomials in $n$ variables by $\R[x_1,\ldots,x_n]$.
\begin{theorem}{\rm\cite{Put}}
\label{Putinar}
Let $g_0,\ldots,g_{m-1}\in \R[x_1,\ldots,x_n]$. Let the set
$$
S=\{x\in\R^n\colon g_i(x)\ge 0,\,i=0,\ldots,m-1\}
$$
be compact and nonempty. If a polynomial $p\in \R[x_1,\ldots,x_n]$ 
is positive on $S$ then there exist an integer $N$ and polynomials
$$
r_i, r_{ij}\in\R[x_1,\ldots,x_n], \quad i=0,\ldots,m-1, \ \ j=0,\ldots,N,
$$ 
such that
\beq
\label{sos}
p = \sum\limits_{j=0}^N r_j^2 + 
\sum\limits_{i=0}^{m-1}\l(\sum_{j=0}^N r_{ij}^2\r)g_i.
\eeq
\end{theorem}
The first result of this type was proved in
\cite{Cas} for the case $m=1$ with $S$ being a disk.
The proof was not constructive and involved Zorn's Lemma. 
In \cite{Put}, Theorem \ref{Putinar} was proved in a similar way. 
In \cite{S} and \cite{NS}, an alternative proof of Theorem \ref{Putinar} 
was presented with its major part being constructive and based on the results 
of \cite{PR}.

In Section 2, we have used Theorem \ref{Putinar} with the polynomials
\beq
\label{g_def}
g_0(x)=1-|x|^2,\quad g_i(x)=|x-\lambda_i|^2-R_i^2,\quad i=1,\ldots,m-1,
\eeq
where $x=(x_1,x_2)$,  $|x|^2=x_1^2+x_2^2$, $\lambda_i\in \R^2$, and $R_i\in\R$.
Let
\beq
\label{s_def}
S=\{x\in\R^2 : g_i(x)\ge 0,\,i=0,\ldots,m-1\}.
\eeq
As before, the set $S$ is a unit disk with several "holes" 
centred at $\lambda_i$ and of radii $R_i$. 

In this section, we give a constructive proof of Theorem \ref{Putinar} for the polynomials \eqref{g_def}.
It turns out that in this case 
the proof simplifies and can be made completely explicit.

If we replace positivity of $p$ with non-negativity,
then for $m=1$ the result still holds. 

\begin{prop}
\label{s_c}
Let $p\in \R[x_1,x_2]$ be non-negative on the unit disk
$\{x\in \R^2 : |x|\leqslant 1\}$. 
Then for some $N$ it admits a representation
$$
p=\sum\limits_{j=0}^N r_j^2+\l(\sum\limits_{j=0}^N s_j^2\r)\l(1-|x|^2\r),
$$
where $r_j,s_j\in \R[x_1,x_2]$, $j=0,\ldots,N$.
\end{prop} 
Proposition \ref{s_c} is a particular case of 
\cite[Corollary 3.3]{Sch}.
We have used it to obtain the representation \eqref{sos_cstar} for the polynomial \eqref{q_def} in Theorem \ref{norm_th}. Note that, in 
contrast with Proposition \ref{s_c}, the condition $p>0$ on $S$ in Theorem 
\ref{Putinar} cannot be replaced by $p\ge 0$ (see Remark 
\ref{counterexample} below).

\subsection{Constructive proof for the polynomials \eqref{g_def}}

The proposed proof relies on the general scheme introduced 
in \cite{S} and \cite{NS} for Theorem \ref{Putinar}. 
We have made all the steps constructive and also added a slight 
variation, the possibility of which was mentioned in \cite{NS}. 
Namely, instead of referring to results of \cite{S} which use \cite{PR}, 
we directly apply the results from \cite{PR} 
(see Proposition \ref{polya} and Lemma \ref{t2} below).

We need the following explicit version of the Lojasiewicz inequality
(see, e.g., \cite{BCR}). Recall that the angle between intersecting circles 
is the minimal angle between their tangents in the intersection points.

\begin{lemma}
\label{loj}
Let $g_0,\ldots,g_{m-1}$ be the polynomials \eqref{g_def}. 
Assume that $S\neq \varnothing$ and none of the disks 
$\{ x : g_i (x) > 0 \}$ with $i>0$
is contained in the union of the others.
Then for any $x\in [-1,1]^2\setminus S$ 
the following estimate holds:
$$
\dist(x,S)\leqslant -c_0 \min\{g_0(x),\ldots,g_{m-1}(x)\}.
$$
If the circles $S_i=\{ x : g_i (x) = 0 \}$
are pairwise disjoint or tangent, then $c_0 = R_{\min}^{-1}$
where $R_{\min}=\min\limits_{i=0,\dots,m-1} R_i$ with $R_0=1$.
Otherwise, $c_0$ can be chosen as
$$
c_0 = \frac{\sqrt{2} + 1}{R_{\min}^2 \sin(\varphi_{\min}/2)} ,
$$
where $\varphi_{\min}$ is the minimal angle between the pairs of intersecting non-tangent circles $S_i$.
\end{lemma}

\noindent We omit the proof of Lemma \ref{loj} because it is elementary 
and involves nothing but school geometry.

For the polynomials
$$
q(x)=\sum\limits_{|\alpha|\leqslant d}q_{\alpha} x^{\alpha}\in \R[x_1,\ldots,x_n],
$$
where $\alpha=(\alpha_1,\ldots,\alpha_n)$ is a multiindex, consider the norm
\beq
\label{pol_norm}
\|q\|=\max\limits_{\alpha}|q_{\alpha}|
\frac{\alpha_1!\ldots\alpha_n!}{(\alpha_1+\ldots+\alpha_n)!}.
\eeq
The following proposition is also elementary and is proved in \cite{NS}:
\begin{prop}
\label{elem2}
Let $x,y\in [-1,1]^n$, $q\in \R[x_1,\ldots,x_n]$, and $\deg q=d$.
Then
$$
|q(x)-q(y)|\leqslant d^2 n^{d-1/2} \|q\||x-y|.
$$
\end{prop}

\noindent The next proposition, which is a quantitative version 
of P\`olya's inequality, is proved in \cite{PR}.

\begin{prop}
\label{polya}
Let $f\in\R[y_1,\ldots,y_{n}]$ be a homogeneous polynomial of degree $d$. 
Assume that $f$ is strictly positive on the simplex
\beq 
\label{sim}
\Delta_n=\{y\in\R^n\colon y_i\geqslant 0,\mbox{ }\sum_i y_i=1\} .
\eeq
Let $f_*=\min\limits_{y\in \Delta_n} f(y)>0$. 
Then, for any
$
N>\frac{d(d-1)\|f\|}{2 f_*}-d,
$
all the coefficients of the polynomial $(y_1+\ldots+y_n)^N f(y_1,\ldots,y_n)$ are positive.
\end{prop}

\label{proof21}

Further on, without loss of generality, we shall be assuming that 
$0\leqslant g_i(x)\le 1$ for all $x\in S$ (if not, we normalize $g_i$ 
multiplying them by positive constants).

\begin{lemma}
\label{t1}
Under the conditions of Theorem {\rm \ref{Putinar}} with $g$ given by \eqref{g_def}, 
let $p^*=\min\limits_{x\in S}p(x)>0$. 
Then
\beq
\label{sums_g}
p(x)-c_0 d^2 2^{d-1/2}\|p\|\sum\limits_{i=0}^{m-1}(1-g_i(x))^{2k}g_i(x)
\geqslant \frac{p^*}{2},\quad \forall x\in [-1,1]^2,
\eeq
where an integer $k$ is chosen in such a way that 
$$
(2k+1) p^* \geqslant m c_0 d^2 2^{d+1/2}\|p\|,
$$ 
and $c_0$ is the constant from Lemma {\rm \ref{loj}}.
\end{lemma}

\begin{proof}
Let $x\in S$. 
Then $p(x)\ge p^*$. Due to our choice of $k$, the elementary inequality
\beq
\label{elem1}
 (1-t)^{2k}t<\frac{1}{2k+1},\quad 0\le t\le 1,\quad k\ge 0,
\eeq
implies that the absolute value of the second term in the left 
hand side of \eqref{sums_g} does not exceed $\frac{p^*}{2}$.

Assume now that $x\in [-1,1]^2\setminus S$. 
Let $y\in S$ be such that $\dist(x,y)=\dist(x,S)$. 
Then Proposition \eqref{elem2} and Lemma \ref{loj} yield
\begin{eqnarray}
\nonumber
p(x)\geqslant p(y)-|p(x)-p(y)|\ge p^*-d^2 2^{d-1/2} \|p\| \dist(x,S) \\
\geqslant p^*+c_0 d^2 2^{d-1/2} \|p\|g_{\min}(x),
\label{26a}
\end{eqnarray}
where $g_{\min}(x)$ is the (negative) minimum of the values of $g_i(x)$. 
Note that $(1-g_{\min}(x))^{2k}>1$. 
From \eqref{26a}, we get
\begin{eqnarray*}
p(x)-c_0 d^2 2^{d-1/2} \|p\|(1-g_{\min}(x))^{2k}g_{\min}(x) \\
\geqslant p(x)-c_0 d^2 2^{d-1/2} \|p\| g_{\min}(x)
\geqslant p^*.
\end{eqnarray*}
On the other hand, \eqref{elem1} and the choice of $k$ imply that the terms with $g_i(x)>0$ contribute no more than
$$
\frac{(m-1) c_0 d^2 2^{d-1/2}\|p\|}{2k+1}\leqslant \frac{p^*}{2}
$$
to the sum \eqref{sums_g}. The remaining terms in \eqref{sums_g} with $g_i(x)<0$ may only increase the left hand side.
\end{proof}

\begin{lemma}
\label{t2}
Let $p\in \R[x_1,x_2]$ and $p_*=\min\limits_{x\in [-1;1]^2}p(x)>0$. 
Then, for some $M\in\N$, 
\beq
\label{t2_res}
p = \sum_{|\alpha|\leqslant M} b_{\alpha}
\gamma_1^{\alpha_1}\gamma_2^{\alpha_2}\gamma_3^{\alpha_3}\gamma_4^{\alpha_4}
\eeq
where $b_{\alpha}\geqslant 0$,
\begin{equation}
\label{gam}
\gamma_1(x)=\frac{1+x_1}4,\quad \gamma_2(x)=\frac{1-x_1}4,\quad \gamma_3(x)=\frac{1+x_2}4,\quad \gamma_4(x)=\frac{1-x_2}4.
\end{equation}
\end{lemma}
\noindent This lemma was obtained in \cite{PR} for arbitrary convex polyhedra and associated linear functions $\gamma_k$. 
Below we prove it for the square $[-1,1]^2$, because in this particular case the formulae are considerably simpler.
\begin{proof}
Consider the following $\R$-algebra homomorphism
$$
\varphi\colon \R[y_1,y_2,y_3,y_4]\to \R[x_1,x_2],\quad y_i\mapsto \gamma_i(x).
$$
In order to prove the lemma, it suffices to find a polynomial $\tilde{p}\in \R[y_1,y_2,y_3,y_4]$ with 
positive coefficients such that $\varphi(\tilde{p})=p$. 
If $p=\sum\limits_{i+j\leqslant d}p_{ij}x_1^i x_2^j$ and
$$
\tilde{p}_1(y)=\sum_{i+j\le d}2^{i+j}p_{ij}(y_1-y_2)^i(y_3-y_4)^j(y_1+y_2+y_3+y_4)^{d-i-j},
$$
then $\varphi(\tilde{p}_1)=p$ because
$$
\varphi(y_1+y_2+y_3+y_4)=1, \quad 2\varphi(y_1-y_2)=x_1, \quad
2\varphi(y_3-y_4)=x_2.
$$
Let 
$$
V=\{y\in \Delta_4\colon 2y_1+2y_2=2y_3+2y_4=1\} ,
$$ 
where $\Delta_4$ is the simplex \eqref{sim}.
If $y\in V$ then 
$\tilde{p}_1(y)=p(4y_1-1,4y_3-1)\ge p_*$, 
as $(4y_1-1,4y_3-1)\in [-1,1]^2$. 
For an arbitrary $y$, let $y_0\in V$ be such that $\dist(y,y_0)=\dist(y,V)$. 
Then, from Proposition \ref{elem2}, 
\beq
\label{tp1}
\tilde{p}_1(y)\ge \tilde p_1(y_0)-|\tilde p_1(y)-\tilde p_1(y_0)|\geqslant p_*-d^2 2^{2d-1} \|\tilde{p}_1\|\dist(y,V).
\eeq
Let
$$
r(y)=2(y_1+y_2-y_3-y_4)^2.
$$
It is easy to see that $\varphi(r)=0$ and 
$$
r(y)=(2y_1+2y_2-1)^2+(2y_3+2y_4-1)^2, \quad \forall y\in \Delta_4.
$$
If we rewrite the last expression in the coordinates
$\frac{y_1+y_2}{\sqrt{2}}$, $\frac{y_1-y_2}{\sqrt{2}}$, 
$\frac{y_3+y_4}{\sqrt{2}}$, $\frac{y_3-y_4}{\sqrt{2}}$ 
(obtained by two rotations by the angle $\pi/4$), then we get
\beq
\label{d4}
r(y)\ge 8 \dist(y,V)^2,\quad \forall y\in \Delta_4.
\eeq
Let
$$
\tilde{p}_2(y)=\tilde{p}_1(y)+
\frac{2^{4d-6}d^4\|\tilde{p}_1\|^2}{p_*}(y_1+y_2+y_3+y_4)^{d-2}r(y).
$$
We still have $\varphi(\tilde{p}_2)=p$. 
The inequalities \eqref{tp1} and \eqref{d4} imply that
\begin{eqnarray*}
\tilde{p}_2(y) \geqslant p_* - d^2 2^{2d-1} \|\tilde{p}_1\| \dist(y,V)
+ \frac{2^{4d-3} d^4 \|\tilde{p}_1\|^2}{p_*} \dist(y,V)^2 = \\
\frac{2^{4d-3}d^4 \|\tilde{p}_1\|^2}{p_*}{}
\l(\dist(y,V)-\frac{p_*}{d^2 2^{2d-1}\|\tilde{p}_1\|}\r)^2 + \frac{p_*}{2}
\geqslant\frac{p_*}{2}, \quad \forall y\in \Delta_4.
\end{eqnarray*}
Finally, since $\tilde{p}_2$ is homogeneous, Proposition \ref{polya} with
$N>\frac{d(d-1)\|\tilde{p}_2\|}{p_*}-d$ shows that all the coefficients of
$$
\tilde{p}(y)=(y_1+y_2+y_3+y_4)^N\tilde{p}_2(y)
$$ 
are positive. 
Applying the homomorphism $\varphi$ to $\tilde{p}$, 
we obtain the desired representation of $p$.
\end{proof}

\noindent{\it End of the proof of Theorem }\ref{Putinar}.
Let us apply Lemma \ref{t1} to $p$. 
It is sufficient to find a representation of 
the left hand side of \eqref{sums_g},
because the second term is already of the form \eqref{sos}.
By Lemma \ref{t2}, the left hand side of \eqref{sums_g} can be 
represented in the form \eqref{t2_res}. 
Note that $\gamma_i$ can be rewritten as
\begin{equation}
\label{gamma_u}
\frac{1}{4}(1\pm x_{1,2})=\frac18\l((1\pm x_{1,2})^2+g_0(x)+x_{2,1}^2\r).
\end{equation}
Substituting the last equality into \eqref{t2_res}, we obtain the desired 
representation for \eqref{sums_g} and, therefore, for $p$.
\qed
\subsection{Some remarks}
\begin{remark}
\label{constr}
If $g_i$ are given by \eqref{g_def} then, in principle, 
it is possible to write down explicit formulae for the polynomials 
appearing in \eqref{sos}. 
Indeed, assume that
we have a polynomial $p$ such that $p(x)\ge p^* > 0$ for all $x\in S$. 
Then
\beq
\label{p1}
p(x)=\hat{p}(x)+c_0 d^2 2^{d-1/2}\|p\|\sum\limits_{i=0}^{m-1}(1-g_i(x))^{2k}g_i(x),
\eeq
where $k$ is chosen in such a way that 
$(2k+1) p^* \geqslant m c_0 d^2 2^{d+1/2}\|p\|$. 
The second term in the right hand side of \eqref{p1} 
is an explicit expression of the form \eqref{sos}, 
and the coefficients of $\hat p$ can be found from \eqref{p1}. 
From Lemma \ref{t1}, we know that $\hat p(x)\ge p^*/2$ for all $x\in [-1;1]^2$. 
Now it suffices to represent 
$$
\hat p(x)=\sum\limits_{k+l\le \hat d}\hat p_{kl}\,x_1^k x_2^l
$$
in the form \eqref{sos}. 
Consider the following polynomials 
$$
\tilde{p}_1(y)=\sum_{i+j\le \hat d}2^{i+j}\hat p_{ij}(y_1-y_2)^i(y_3-y_4)^j(y_1+y_2+y_3+y_4)^{\hat d -i-j},
$$
$$
\tilde{p}_2(y)=\tilde{p}_1(y)+
\frac{2^{4\hat d -4}\hat d^4\|\tilde{p}_1\|^2}{p^*}
(y_1+y_2+y_3+y_4)^{\hat d -2} (y_1+y_2-y_3-y_4)^2,
$$
and
$$
\tilde{p}(y)=(y_1+y_2+y_3+y_4)^N\tilde{p}_2(y)
\quad \text{where} \quad 
N>\frac{2\hat d (\hat d - 1) \|\tilde{p}_2\|}{p^*}- \hat d.
$$ 
If we replace $y_i$, $i = 1, 2, 3, 4$, 
with $\gamma_i(x)$ given by \eqref{gam} in the definition of $\tilde{p}$,
then we get $\hat p(x)$. 
The coefficients of $\tilde{p}$ are positive. 
Therefore, if we substitute $y_i$ with $\gamma_i$ 
and then apply \eqref{gamma_u}, 
we obtain an expression of the form \eqref{sos} for $\hat p(x)$. 
Combining it with $\eqref{p1}$, we get the desired expression for $p$. 
As a consequence, if we have a continuous family of positive polynomials 
with a uniform lower bound on $S$ and uniformly bounded degrees, then the 
polynomials in the representation \eqref{sos} may also be chosen to be 
continuously depending on this parameter, and also with uniformly bounded degrees.
\end{remark}

\begin{remark}
\label{counterexample}
In \cite{Sch2}, an analogue of Theorem \ref{Putinar} for a non-negative 
polynomial $p$ and $m>1$ was established under some additional assumptions 
on the zeros of $p$. The next theorem 
shows that, in general, Theorem \ref{Putinar} may not be true if $p\ge 0$.
\end{remark}

\begin{theorem}
\label{counter}
Let $g_i$ be defined by \eqref{g_def},
and assume that $\lambda_i \neq \lambda_j$ for some $i$ and $j$. 
Then the polynomial $g_i g_j$ can not be represented in the form \eqref{sos}.
\end{theorem}
This result is probably well known to specialists, 
although we could not find it in the literature. 
For reader's convenience, we prove it below.

Let $g_i$ be defined by \eqref{g_def}, and let
\beq
\label{26}
S_i=\{x\in\R^2 : g_i(x)=0\} , \quad
S_i(\C)=\{x\in\C^2 : g_i(x)=0\}.
\eeq

\begin{lemma}
\label{hilbert}
Let $q\in \R[x_1,x_2]$ be a polynomial such that $q(x)=0$ on an open arc of $S_i$. 
Then $g_i\mid q$  {\rm(}that is, $q$ is divisible by $g_i${\rm)}.
\end{lemma}

\begin{proof}
Consider $q$ as an analytic function on $S_i(\C)$. 
Since the set $S_i(\C)$ is connected, $q\equiv 0$ on the whole $S_i(\C)$. 
Hilbert's Nullstellensatz (see, for example, \cite[Section 16.3]{Var}) 
gives that $g_i\mid q^k$ 
for some integer $k$ (in $\C[x_1,x_2]$ and, consequently, in $\R[x_1,x_2]$). 
As the polynomial $g_i$ is irreducible, we have $g_i\mid q$.
\end{proof}

\begin{lemma}
\label{empty}
Let 
$\lambda_i \neq \lambda_j$.
Then $S_i(\C)\cap S_j(\C)\neq\varnothing$.
\end{lemma}

\begin{proof}
Let the circles $S_i$ and $S_j$ be given by the equations
$$
(x_1-a_1)^2+(x_2-a_2)^2=R_1^2,\quad (x_1-b_1)^2+(x_2-b_2)^2=R_2^2.
$$
Subtracting one from the other, we get a system of a linear and 
a quadratic equation. 
The linear one is solvable because $\lambda_i\neq \lambda_j$. 
Substituting the solution into the quadratic equation, we reduce it to 
a non-degenerate quadratic equation in 
one complex variable, which also has a solution.
\end{proof}

\noindent{\it Proof of Theorem {\rm \ref{counter}}.}
Assume that $p=g_i g_j$ satisfies $\eqref{sos}$.
The left hand side of $\eqref{sos}$ vanishes on the set $S_i \cap \partial S$. 
All the terms $r_k^2$ and $r_{kl}^2 g_k$ in the right hand side of \eqref{sos}
are non-negative on $S_i \cap \partial S$, and therefore are equal to zero 
on this set.
By Lemma \ref{hilbert}, they all are multiples of $g_i$.
Similarly, all the terms in the right hand side are multiples of $g_j$. 
Therefore, $g_i\mid r_k$, $g_j\mid r_k$, and $g_i^2 g_j^2\mid r_k^2$.

Since the polynomials $g_k$ and $g_i$ are coprime for all  $k\neq i$,
we have $g_i^2 \mid r_{kl}^2$ for $k\neq i$
and $g_j^2\mid r_{kl}^2$ for $k\neq j$.
Thus any term in the right hand side of \eqref{sos} is a multiple of either 
 $g_i^2 g_j$ or $g_i g_j^2$. 
Dividing \eqref{sos} by
$g_i g_j$, we see that the left hand side is identically equal to $1$,
and the right hand side vanishes on the intersection
$S_i(\C)\cap S_j(\C)$ which is nonempty by Lemma \ref{empty}.
This contradiction proves the theorem. \qed


\end{document}